\documentclass[twoside,11pt]{article}

\usepackage{jmlr2e-arxiv}
\usepackage{amsmath}

\newcommand{\pa}{\mathrm{pa}}
\newcommand{\nei}{\mathrm{ne}}

\newcommand{\an}{\mathrm{an}}

\newcommand{\ant}{\mathrm{ant}}

\newcommand{\dse}{\,\mbox{$\perp$}\,}

\newcommand{\cip}{\mbox{\,$\perp\!\!\!\perp$\,}}
\newcommand{\sk}{\mathrm{sk}}

\newcommand{\cd}{\,|\,}

\newcommand{\nls}{\vspace{-1mm}\\}

\newcommand{\mtp}{{\rm MTP}_{2}}
\newcommand{\ci}{\mbox{\protect $\: \perp \hspace{-2.3ex}
\perp$ }}

\newcommand{\notdse}{\nolinebreak{\not\hspace{-1.5mm}\dse}}

\newcommand{\notci}{\nolinebreak{\not\hspace{-1.5mm}\ci}}

\newcommand{\n}[0]{\hspace*{.35em}}

\newcommand{\nn}[0]{\hspace*{.7em}}

\newcommand{\ful}{\mbox{$\, \frac{ \nn \nn \;}{ \nn \nn
}$}}

\newcommand{\fla}{\mbox{$\hspace{.05em} \prec
\!\!\!\!\!\frac{\nn \nn}{\nn}$}}

\newcommand{\fra}{\mbox{$\hspace{.05em} \frac{\nn
\nn}{\nn
}\!\!\!\!\! \succ \! \hspace{.25ex}$}}

\newcommand{\arc}{\mbox{$\hspace{.06em} \prec
\!\!\!\!\!\frac{\nn \nn}{\nn}
\!\!\!\!\!
\succ\! \hspace{.25ex}$}}

\begin{document}

\title{Faithfulness of Probability Distributions and Graphs}

\author{\name Kayvan Sadeghi \email k.sadeghi@statslab.cam.ac.uk \\
       \addr Statistical Laboratory\\
       University of Cambridge\\
       Centre for Mathematical Sciences, Wilberforce Road\\
       Cambridge, CB3 0WB, United Kingdom}


\maketitle

\begin{abstract}
 A main question in graphical models and causal inference is whether, given a probability distribution $P$ (which is usually an underlying distribution of data), there is a graph (or graphs) to which $P$ is faithful. The main goal of this paper is to provide a theoretical answer to this problem. We work with general independence models, which contain probabilistic independence models as a special case. We exploit a generalization of ordering, called preordering, of the nodes of (mixed) graphs. This allows us to provide sufficient conditions  for a given independence model to be Markov to a graph with the minimum possible number of edges, and more importantly,  necessary and sufficient conditions for a given probability distribution to be faithful to a graph.
We present our results for the general case of mixed graphs, but specialize the definitions and results to the better-known subclasses of undirected (concentration) and bidirected (covariance) graphs as well as directed acyclic graphs.
\end{abstract}


\section{Introduction}
Graphs have been used in graphical models in order to capture the conditional independence structure of probability distributions. Generally speaking, nodes of the graph correspond to random variables and edges to conditional dependencies \citep{lau96}. The connection between graphs and probability distributions is usually established in the literature by the concept of Markov property \citep{cli90,pea88,stu89}, which ensures that if there is a specific type of separation between nodes $i$ and $j$ of the graph ``given the node subset $C$" then random variables $X_i$ and $X_j$ are conditionally independent given the random vector $X_C$ in the probability distribution. However, the ``ultimate" connection between probability distributions and graphs requires the other direction to hold, namely for every conditional independence in the probability distribution to correspond to a separation in the graph. This connection has been called faithfulness of the probability distribution and the graph in \citet{spioo}, and the graph has been called the perfect map of such a distribution in \citet{pea88}.

However, ``given a probability distribution $P$ whether there is a graph (or graphs) to which $P$ is faithful" is an open problem, and consequently so is the problem of finding these graphs. This problem can be raised for any type of graph existing in the literature of graphical models ranging from different types of mixed graphs with three types of edges \citep{ric02,wer11,sad16} and different types of chain graphs \citep{lau89,fry90,and01,cox93,wer11} to better-known classes of undirected (concentration) \citep{dar80} and bidirected (covariance) \citep{kau96} graphs as well as directed acyclic graphs \citep{kii84,pea88}. Our goal is to provide an answer to this problem. A similar problem of ``given a graph whether there is a family of distributions faithful to it" has been answered for very specific types of graphs and specific types of distributions; for example, for Gaussian distributions and undirected graphs in \citet{lne07}, DAGs in \citet{pea88,gei90,spioo}, ancestral graphs in \citet{ric02}, and LWF chain graphs in \citet{pen11}; and discrete distributions and DAGs in \citet{gei90,mee95} and LWF chain graphs in \citet{pen09}.

The concept of faithfulness was originally defined for the purpose of causal inference \citep{spioo,pea88}, and the theory developed in this paper can be interpreted in causal language.  A main approach to causal inference is based on graphical representations of causal structures, usually represented by \emph{causal graphs} that are directed acyclic with nodes as random variables (that is a Bayesian network). Causal graphs are assumed to capture the true causal structure; see, for example, \citet{nea04}. A main assumption made is called the \emph{causal faithfulness condition} (CFC) stating that the true probability distribution is faithful to the true causal graph. Although the
distribution being faithful to the ``true underlying causal graph'' is a much stronger
assumption, it necessarily means that it must be faithful to ``some graph'' in order for one to be able to use graphical methods for causal inference.  For an extensive discussion on the CFC, see \citet{zha08}, and for related philosophical discussions, see, for example, \citet{woo98,ste06}. Although the results in this paper cover those of causal Bayesian networks, we have developed our theory to a much more general case of simultaneous representation of ``direct effects", ``confounding", and ``non-causal symmetric dependence structures"; see, for example, \citet{pea09}.

In this paper, we work with the general independence model $\mathcal{J}$. Independence models $\mathcal{J}(P)$ induced by probability distributions $P$ are a special case. We provide necessary and sufficient conditions for $\mathcal{J}$ to be graphical, that is to be faithful to a graph. In short, as proved in Theorem \ref{thm:2}, $\mathcal{J}$ is graphical if and only if it satisfies the so-called compositional graphoid axioms as well as singleton-transitivity, and what we call ordered upward- and downward-stability.

As apparent from their names, ordered upward- and downward-stability depend on a generalization of ordering of variables, and consequently the nodes of the graph (called preordering). We provide our results for the most general case of mixed graphs, which contains almost all classes of graphs used in graphical models as  subclasses with the exception of the chain graphs with the AMP Markov property \citep{and01} and its generalizations \citep{pen14}. However, based on the preordering w.r.t.\ which ordered upward- and downward-stability are satisfied, one can deduce to what type (or types) of graph $P$ is faithful.

Since many of the subclasses of mixed graphs are not as well-known or used as some simpler classes of graphs, we provide a specialization of definitions and the main results for undirected and bidirected graphs as well as directed acyclic graphs at the end of the paper. We advise the readers that are only familiar with or interested in these simpler subclasses to skip the general definitions and results for mixed graphs, and focus on the specialization.

The structure of the paper is as follows: In the next section, we provide the basic definitions for graphs as well as different classes of graphs in graphical models, and independence models. In Section \ref{sec:prob}, we define and provide basic results for Markovness and faithfulness of independence models and graphs. In Section \ref{sec:order}, we define and exploit the preordering for (so-called anterial) graphs, and then we define the concept of upward- and downward-stability. In Section \ref{sec:charac}, we provide sufficient conditions for an independence model to be so-called minimally Markov to a graph, and then we provide necessary and sufficient conditions for an independence model to be faithful to a graph. 
In Section \ref{sec:spec}, we specialize the definitions and results to the classes of undirected and bidirected graphs as well as directed acyclic graphs. In Section \ref{sec:res}, we show  how other results in the literature concerning faithfulness of certain probability distributions and certain classes of graphs are corollaries of our results. This also provides a nice set of examples for the theory presented in the paper. In Section \ref{sec:dis}, we end the paper with a short summary, a discussion on the statistical implications of the results, and an outline of future work.
\section{Definitions and Concepts}
In this section, we provide the basic definitions and concepts needed for the paper.
\subsection{Graph Theoretical Definitions and Concepts}
A \emph{graph} $G$ is a triple consisting of a \emph{node} set or
\emph{vertex} set $V$, an \emph{edge} set $E$, and a relation that with
each edge associates two nodes (not necessarily distinct), called
its \emph{endpoints}. When nodes $i$ and $j$ are the endpoints of an
edge, they are
\emph{adjacent}. We call a node adjacent to the node $i$, a \emph{neighbour} of $i$, and denote the set of all neighbours of $i$ by $\nei(i)$. We say that an edge is \emph{between} its two
endpoints. We usually refer to a graph as an ordered
pair $G=(V,E)$. Graphs $G_1=(V_1,E_1)$ and $G_2=(V_2,E_2)$ are called \emph{equal} if $(V_1,E_1)=(V_2,E_2)$. In this case we write $G_1=G_2$.

Notice that our graphs are \emph{labeled}, that is every node is considered a different object. Hence, for example, graph $i\ful j\ful k$ is not equal to $j\ful i\ful k$.

A \emph{loop} is an edge with
the same endpoints. Here we only discuss graphs without loops. \emph{Multiple edges} are edges with the
same pair of endpoints. A \emph{simple graph} has neither
loops nor multiple edges.

Graphs in this paper are so-called \emph{mixed graphs}, which contain three types of edges: \emph{lines} ($i\ful j$), \emph{arrows} ($i\fra j$), and \emph{arcs} ($i\arc j$), where in the two latter cases, we say that there is an \emph{arrowhead} at $j$. If we remove all arrowheads from edges of the graph  $G$, the obtained undirected graph is called the \emph{skeleton} of $G$. The simple graph whose edge $ij$ indicates whether there is an edge (or multiple edges) between $i$ and $j$ in $G$ is called the \emph{adjacency graph} of $G$. It is clear that if a graph is simple then its adjacency graph is the same as its skeleton. In this paper, only skeleton of anterial graphs or its subclasses are used, which, as will be defined later, are simple graphs; hence, we denote the  simple skeleton by $\sk(G)$. We say that we \emph{direct} the edges of a skeleton by putting arrowheads at the edges in order to obtain mixed graphs.

A \emph{subgraph} of a graph $G_1$ is graph $G_2$ such that $V(G_2)\subseteq V(G_1)$ and $E(G_2)\subseteq E(G_1)$ and the assignment of endpoints to edges in $G_2$ is the same as in $G_1$.
We define a subgraph \emph{induced by edges} $A\subseteq E$ of $G=(V,E)$ to be a subgraph that contains $V$ as the node set and all and only edges in $A$ as the edge set.

A \emph{walk} is a list $\langle v_0,e_1,v_1,\dots,e_k,v_k\rangle$ of nodes and edges such that for $1\leq i\leq k$, the edge $e_i$ has endpoints $v_{i-1}$ and $v_i$. A \emph{path} is a walk with no repeated node or edge. A maximal set of nodes in a graph whose members are connected by some paths constitutes a \emph{connected component} of the graph. A \emph{cycle} is a walk  with no repeated nodes or edges except for $v_0=v_k$.


We say a walk is \emph{between} the first and the last nodes of the list in $G$. We
call the first and the last nodes \emph{endpoints} of the walk and all other nodes \emph{inner nodes}.

%

A \emph{subwalk} of a walk $\omega=\langle i_0,e_1,i_1,\dots,e_n,i_n\rangle$ is a walk $\langle i_r,e_{r+1},i_{r+1},\dots,e_p,i_p\rangle$ that is a subsequence  of $\omega$ between two occurrences of nodes ($i_r,i_p$, $0\leq r\leq p\leq n$). If a subwalk forms a path then it is called a \emph{subpath} of $\omega$.

A walk $\omega=\langle i=i_0,i_1,\dots,i_n=j\rangle$ is \emph{directed} from $i$ to $j$ if all  edges $i_ki_{k+1}$, $0\leq k\leq n-1$, are arrows pointing from $i_k$ to $i_{k+1}$. If there is a directed walk from $i$ to $j$ then $i$ is an \emph{ancestor} of $j$ and $j$ is a \emph{descendant} of $i$. We denote the set of ancestors of $j$ by $\an(j)$.

A walk $\omega=\langle i=i_0,i_1,\dots,i_n=j\rangle$ is \emph{semi-directed} from $i$ to $j$  if it has at least one arrow, no arcs, and every arrow $i_ki_{k+1}$ is pointing from $i_k$ to $i_{k+1}$.  A walk between $i$ and $j$ is \emph{anterior} from $i$ to $j$ if it is semi-directed from $i$ to $j$ or  if it is undirected. If there is an anterior walk from $i$ to $j$ then we also say that $i$ is an \emph{anterior} of $j$. We use
the notation $\ant(j)$ for the set of all anteriors of $j$. For a set $A$, we define $\ant(A)=\bigcup_{j\in A}\ant(j)\setminus A$.

Notice that, unlike in most places in the literature (for example  \citet{ric02}), we use walks instead of paths to define ancestors and anteriors. Because of this and the fact that ancestral graphs have no arrowheads pointing to lines, our definition of anterior extends the notion of anterior for ancestral graphs in  \citet{ric02} with the modification that in this paper, a node is not an anterior of itself. Using walks instead of paths is immaterial, as shown in \citet{sadl16}.

%
%

A \emph{section} $\rho$ of a walk  is a maximal subwalk consisting only of lines, meaning that there is no other subwalk that only consists of lines and includes $\rho$. Thus, any walk decomposes uniquely into sections; these are not necessarily edge-disjoint and sections may also be single nodes. 
A section $\rho$ on a walk $\omega$ is called a  \emph{collider section} if one of the following walks is a subwalk of $\omega$: $i\fra\rho\fla\,j$, $i\arc\rho\fla\,j$, $i\arc\rho\arc\,j$. All other sections on $\omega$ are called \emph{non-collider} sections.  Notice that a section may be a collider on one walk and a non-collider on another, but we may speak of collider or non-collider sections without mentioning the relevant walk when this is apparent from the context.

%
\subsection{Different Classes of Graphs}
All the subclasses of graphs included here are \emph{acyclic} in the sense that they do not contain semi-directed cycles. The most general class of graphs discussed here is the class of \emph{chain mixed graphs} (CMGs) \citep{sad16} that contains all mixed graphs without semi-directed cycles. They may have multiple edges containing an arc and a line or an arc and an arrow but not a combination of an arrow and a line or arrows in opposite directions, as such combinations would constitute semi-directed cycles. In this paper, in using the term ``graph'' we mean a CMG unless otherwise stated.

A general class of graphs that plays an important role in this paper is the class of \emph{anterial graphs} (AnGs) \citep{sad16}. AnGs are CMGs in which an endpoint of an arc cannot be an anterior of the other endpoint.

CMGs include \emph{summary graphs} (SGs) \citep{wer11} and \emph{acyclic directed mixed graphs} (ADMGs) \citep{ric03} as a subclass, but not  \emph{AMP chain graphs} \citep{and01}, and AnGs include \emph{ancestral graphs} (AGs) \citep{ric02} but not SGs or ADMGs. These have typically been introduced to describe independence structures obtained by marginalization and conditioning in DAG independence models; see for example \citet{sad13}.

AnGs also contain \emph{undirected and bidirected chain graphs} (CGs). A chain graph is an acyclic graph  so that if we remove all arrows, all connected components of the resulting graph -- called \emph{chain components} -- contain one type of edge only.  If all chain components contain lines, the chain graph is an undirected chain graph (UCG) (known as \emph{LWF chain graphs}); if all chain components contain arcs  it is a bidirected chain graph (BCG) (known as \emph{multivariate regression chain graphs}). AnGs also contain \emph{Regression graphs} \citep{wers11}, which are chain graphs consisting of lines and arcs (although dashed undirected edges have mostly been used instead of arcs in the literature), where there is no arrowhead pointing to lines.


These also contain graphs with only one type of edge; namely \emph{undirected graphs} (UGs), containing only lines;  \emph{bidirected graphs} (BGs), containing  only bidirected edges; and \emph{directed acyclic graphs} (DAGs), containing only arrows and being acyclic. Clearly, a graph without arrows has no semi-directed cycles, and a semi-directed cycle in a graph with only arrows is a directed cycle. \citet{cox93,kau96,wer98,drt08}  used the terms \emph{concentration graphs} and  \emph{covariance graphs} for UGs and BGs, referring to their independence interpretation associated with covariance and concentration matrices for Gaussian graphical models. DAGs have been particularly useful  to describe causal Markov relations; see for example \citet{kii84,pea88,lau88,gei90,spioo}. 

For an extensive discussion on the subclasses of acyclic graphs and their relationships and hierarchy, see \citet{sadl16}.

\subsection{Independence Models and Their Properties}\label{sec:propind}
An \emph{independence model} $\mathcal{J}$ over a finite set $V$ is a set of triples $\langle X,Y\cd Z\rangle$ (called \emph{independence statements}), where $X$, $Y$, and $Z$ are disjoint subsets of $V$; $Z$
may be empty, but $\langle \varnothing,Y\cd Z\rangle$ and $\langle X,\varnothing\cd Z\rangle$ are always included in $\mathcal{J}$. The independence statement $\langle X,Y\cd Z\rangle$ is read as ``$X$ is independent of $Y$ given $Z$''. Independence models may in  general have a  probabilistic interpretation, but not necessarily. Similarly, not all independence models can be easily represented by graphs. For further discussion on general independence models, see \citet{stu05}.

In order to define probabilistic independence models, consider a set $V$ and a collection of random variables
$\{X_\alpha\}_{\alpha\in V}$ with state spaces $\mathcal{X}_\alpha, \alpha\in V$ and joint distribution $P$. We let $X_A=\{X_v\}_{v\in A}$ etc.\ for each subset $A$ of $V$. For disjoint subsets $A$, $B$, and $C$ of $V$
we use the short notation $A\cip B\cd C$ to denote that $X_A$ is \emph{conditionally independent of $X_B$ given $X_C$} \citep{daw79,lau96}, that is that for any measurable $\Omega\subseteq \mathcal{X}_A$ and $P$-almost all $x_B$ and $x_C$,
$$P(X_A \in \Omega\cd X_B=x_B, X_C=x_C)=P(X_A \in \Omega\cd X_C=x_C).$$
We can now induce an independence model $\mathcal{J}(P)$ by letting
\begin{displaymath}
\langle A,B\cd C\rangle\in \mathcal{J}(P) \text{ if and only if } A\ci B\cd C \text{ w.r.t.\ $P$}.
\end{displaymath}
Similarly we use the notation $A\notci B\cd C$ for $\langle A,B\cd C\rangle\notin \mathcal{J}(P)$.

In order to use graphs to represent independence models, the notion of \emph{separation} in a graph is fundamental. For three disjoint subsets $A$, $B$, and $C$, we use the notation $A\dse B\cd C$ if $A$ and $B$ are \emph{separated} given $C$, and $A\notdse B\cd C$ for $A$ and $B$ not separated given $C$.

The separation can be unified under one definition for all graphs by using walks instead of paths: A walk $\pi$ is \emph{connecting}  given $C$ if every  collider section of $\pi$ has a node in $C$ and all non-collider sections are disjoint from $C$. For pairwise disjoint subsets $(A,B,C)$, $A\dse B\cd C$ if there are no connecting walks between $A$ and $B$ given $C$. This is in wording the same as the definition in \citet{stub98}  for undirected (LWF) chain graphs (although it is in fact a generalization as collider sections are generalized), and a generalization of the \emph{$m$-separation} used with different wordings in \citet{ric02,wer11,wers11} for AGs and SGs, and of the  \emph{d-separation} of \citet{pea88}. For UGs, the notion has a direct intuitive meaning, so that $A\dse B \cd C$ if all paths from $A$ to $B$  pass through $C$.

If $A$, $B$, or $C$ has only one member $\{i\}$, $\{j\}$, or $\{k\}$, for better readability, we write $\langle i,j\cd k\rangle\in \mathcal{J}$ instead of $\langle \{i\},\{j\}\cd \{k\}\rangle\in \mathcal{J}$; and similarly for $i\dse j\cd k$  and $i\ci j\cd k$. We also write $A\dse B$ when $C=\varnothing$; and similarly $A\ci B$.


A graph $G$ induces an independence model $\mathcal{J}(G)$ by separation, letting
\begin{displaymath}
\langle A,B\cd C\rangle\in \mathcal{J}(G)\iff A\dse B\cd C \text{ in } G.
\end{displaymath}

An independence model $\mathcal{J}$ over a set $V$ is a \emph{semi-graphoid} if it satisfies the four following properties for disjoint subsets $A$, $B$, $C$, and $D$ of $V$:
 \begin{enumerate}
    \item $\langle A,B\cd C\rangle\in \mathcal{J}$ if and only if $\langle B,A\cd C\rangle\in \mathcal{J}$ (\emph{symmetry});
    \item if $\langle A,B\cup D\cd C\rangle\in \mathcal{J}$ then $\langle A,B\cd C\rangle\in \mathcal{J}$ and $\langle A,D\cd C\rangle\in \mathcal{J}$ (\emph{decomposition});
    \item if $\langle A,B\cup D\cd C\rangle\in \mathcal{J}$ then $\langle A,B\cd C\cup D\rangle\in \mathcal{J}$ and $\langle A,D\cd C\cup B\rangle\in \mathcal{J}$ (\emph{weak union});
    \item if $\langle A,B\cd C\cup D\rangle\in \mathcal{J}$ and $\langle A,D\cd C\rangle\in \mathcal{J}$
    then $\langle A,B\cup D\cd C\rangle\in \mathcal{J}$ (\emph{contraction}).
 \end{enumerate}
Notice that the reverse implication of contraction clearly holds by decomposition and weak union. A semi-graphoid for which the reverse implication of the weak union property holds is said to be a \emph{graphoid}; that is, it also satisfies
\begin{itemize}
	\item[5.] if $\langle A,B\cd C\cup D\rangle\in \mathcal{J}$ and $\langle A,D\cd C\cup B\rangle\in \mathcal{J}$ then $\langle A,B\cup D\cd C\rangle\in \mathcal{J}$ (\emph{intersection}).
\end{itemize}
Furthermore, a graphoid or semi-graphoid for which the reverse implication of the decomposition property holds is said to be \emph{compositional}, that is, it also satisfies
\begin{itemize}
	\item[6.] if $\langle A,B\cd C\rangle\in \mathcal{J}$ and $\langle A,D\cd C\rangle\in \mathcal{J}$ then $\langle A,B\cup D\cd C\rangle\in \mathcal{J}$ (\emph{composition}).
\end{itemize}
Separation in graphs satisfies all these properties; see Theorem 1 in \citet{sadl16}:
\begin{proposition}\label{prop:110}
For any graph $G$, the independence model $\mathcal{J}(G)$ is a compositional graphoid.
\end{proposition}

On the other hand, probabilistic independence models are always semi-graphoids \citep{pea88}, whereas the converse is not necessarily true; see \citet{stu89}. If, for example, $P$ has strictly positive density, the induced independence model is always a graphoid; see, for example, Proposition 3.1 in \citet{lau96}. See also \citet{pet15} for a necessary and sufficient condition for the intersection property to hold. If the distribution $P$ is a regular multivariate Gaussian distribution, $\mathcal{J}(P)$ is a compositional graphoid; for example see \citet{stu05}.
Probabilistic independence models with positive densities are not in general compositional; this only holds for special types of multivariate distributions such as, for example,  Gaussian distributions and the symmetric binary distributions used in \citet{wer09}.

Another important property that is satisfied by separation in all graphs, but not necessarily for probabilistic independence models, is \emph{singleton-transitivity} (also called \emph{weak transitivity} in \citet{pea88}, where it is shown that for Gaussian and binary distributions $P$, $\mathcal{J}(P)$ always satisfies it). For $i$, $j$, and $k$, single elements in $V$,
 \begin{itemize}
	\item[7.] if $\langle i,j\cd C\rangle\in \mathcal{J}$ and $\langle i,j\cd C\cup\{k\}\rangle\in \mathcal{J}$ then $\langle i,k\cd C\rangle\in \mathcal{J}$ or $\langle j,k\cd C\rangle\in \mathcal{J}$ (singleton-transitivity).
\end{itemize}
A singleton-transitive compositional graphoid is equivalent to what is called a \emph{Gaussoid} in \citet{lne07} with a rather different axiomatization. The name reminds one that these are the axioms satisfied by the independence model of a regular Gaussian distribution \citep{pea88}.
\begin{proposition}\label{prop:1}
For a graph $G$, $\mathcal{J}(G)$ satisfies singleton-transitivity.
\end{proposition}
\begin{proof}
If there is a walk between $i$ and $k$ given $C$ and a walk between $j$ and $k$ given $C$ then by connecting these two walks, we obtain a walk $\pi$ between $i$ and $j$. Except the section $\rho$ that contains $k$, all collider sections on $\pi$ have a node in $C$ and all non-collider sections are outside $C$. Hence depending on $\rho$ being a collider or non-collider, $\pi$ is connecting given $C\cup\{k\}$ or $C$ respectively.
\end{proof}

\section{Markov and Faithful Independence Models}\label{sec:prob}
In this section, we define and discuss the concepts of Markovness and faithfulness for probability distributions, independence models, and graphs.
\subsection{Markov Properties}\label{sec:marpr}
For a graph $G=(V,E)$, an independence model $\mathcal{J}$ defined over $V$ satisfies the \emph{global Markov property} w.r.t.\ $G$, or is simply \emph{Markov} to $G$, if for  disjoint subsets $A$, $B$, and $C$ of $V$ it holds that  $$A\dse B\cd C \implies \langle A,B\cd C\rangle\in \mathcal{J},$$
or equivalently $\mathcal{J}(G)\subseteq\mathcal{J}$. In particular, the graphical independence $\mathcal{J}(G)$ is trivially Markov to its graph $G$.

For a probability distribution $P$, we simply say $P$ is \emph{Markov to} a graph $G$ if $\mathcal{J}(P)$ is Markov to $G$, that is if $\mathcal{J}(G)\subseteq\mathcal{J}(P)$.
%
Notice that every independence model over $V$ is Markov to the complete graph with the node set $V$.

For a graph $G=(V,E)$, an independence model $\mathcal{J}$ defined over $V$ satisfies a \emph{pairwise Markov property} w.r.t.\ a graph $G$, or is simply \emph{pairwise Markov to} $G$, if for  every non-adjacent pair of nodes $i,j$, it holds that  $ \langle i,j\cd C(i,j)\rangle\in \mathcal{J}$, for some $C(i,j)$, where $C(i,j)$ is \emph{the conditioning set} of the pairwise Markov property and does not include $i$ and $j$.

The independence model $\mathcal{J}(G)$ only satisfies pairwise Markov properties  w.r.t.\ what is called \emph{maximal graphs}, which are graphs where the lack of an edge between $i$ and $j$ corresponds to a conditional separation statement for $i$ and $j$. For a probability distribution $P$, we say $P$ is \emph{pairwise Markov to} a graph $G$ if $\mathcal{J}(P)$ is pairwise Markov to $G$.

For maximal graphs, a pairwise Markov property is defined by letting $C(i,j)=\ant(i)\cup \ant(j)\setminus\{i,j\}$ \citep{sadl16}, which we henceforth use as ``the'' pairwise Markov property. The conditioning set of the pairwise Markov property simplifies for DAGs, $C(i,j)=an(i)\cup \an(j)\setminus\{i,j\}$ \citep{ric02,sadl14};  for BGs, $C(i,j)=\varnothing$ (as defined in \citet{wer98}); and for connected UGs, $C(i,j)=V\setminus\{i,j\}$ (as defined in \citet{lau96}).

This pairwise Markov property, is in fact equivalent to the global Markov property under compositional graphoids; see Theorem 4 of \citet{sadl16}:
\begin{proposition}\label{prop:114}
Let $G$ be a maximal graph. If the independence model $\mathcal{J}$  is a compositional graphoid, then $\mathcal{J}$ is pairwise Markov to $G$ if and only if $\mathcal{J}$ is Markov to $G$.
\end{proposition}
In particular, for UGs, the equivalence of the global and pairwise Markov properties holds under graphoids, and for BGs under compositional semi-graphoids.



Two undirected or two bidirected graphs $G$ and $H$, where $G\neq H$, induce different independence models; that is $\mathcal{J}(G)\neq\mathcal{J}(H)$. This is not necessarily true for larger subclasses. We call two graphs $G$ and $H$ such that $\mathcal{J}(G)=\mathcal{J}(H)$ \emph{Markov equivalent}. Conditions for Markov equivalence for most subclasses of graphs are known; see \citet{ver90,ali09,wers11}. Notice that two Markov equivalent maximal graphs have the same skeleton.

\subsection{Faithfulness and Minimal Markovness}
We say that an independence model $\mathcal{J}$ and a probability distribution $P$ are \emph{faithful} if $\mathcal{J}= \mathcal{J}(P)$. Similarly, we say that $\mathcal{J}$ and a graph $G$ are faithful if $\mathcal{J}= \mathcal{J}(G)$. We also say that $P$ and $G$ are faithful if $\mathcal{J}(P)= \mathcal{J}(G)$. If $P$ and $G$ are faithful then we may sometimes also say that $P$ is \emph{faithful to}  $G$ or vice versa, although in principle faithfulness is a symmetric relation; the same holds for saying that $\mathcal{J}$ is faithful to $G$ or $P$. Notice that we are extending the definition of faithfulness to include the relation between  independence models and graphs as well as independence models and probability distributions rather than only that between graphs and probability distributions as originally used in \citet{spioo}.

Thus, if $\mathcal{J}$ and $G$ are faithful then $\mathcal{J}$ is Markov to $G$, but it also requires every independence statement to correspond to a separation in $G$. We say that $\mathcal{J}$ or $G$ is \emph{probabilistic} if there is a distribution $P$ that is faithful to $\mathcal{J}$ or  $G$ respectively. If there is a graph that is faithful to $\mathcal{J}$ or $P$ then we say that $\mathcal{J}$ or $P$ is \emph{graphical}.

Our main goal is to characterize graphical probability distributions, and in addition, if existent, to provide graphs that are faithful to a given $P$. We solve this problem for the general case of independence models $\mathcal{J}$.

For a given independence model $\mathcal{J}$, we define the \emph{skeleton} of $\mathcal{J}$, denoted by $\sk(\mathcal{J})$, to be the undirected graph that is obtained from $\mathcal{J}$ as follows:  we define the node set of $\sk(\mathcal{J})$ to be $V$, and for every pair of nodes $i,j$, we check whether $\langle i,j\cd C\rangle\in\mathcal{J}$  holds for some $C\subseteq V\setminus\{i,j\}$; if it does not then we draw an edge between $i$ and $j$. One can similarly define $\sk(P)$ for a probability distribution $P$ by checking whether $i\ci j\cd C$. This is the same as the graph obtained by the first step of the SGS algorithm \citep{gly87}. Notice that this is not necessarily the same as the undirected graph obtained by the pairwise Markov property for UGs; see Section \ref{sec:un-bi}.


For UGs and BGs, the skeleton of $\mathcal{J}$ uniquely determines the graph, whereas  for other subclasses, there are several graphs with the same skeleton $\sk(\mathcal{J})$. As will be seen, the preordering, defined in the next section, enables us to direct the edges of the skeleton.
%

\begin{proposition}\label{prop:0}
If an independence model $\mathcal{J}$ is Markov to a graph $G$ then $\sk(\mathcal{J})$ is a subgraph of $\sk(G)$.
\end{proposition}
\begin{proof}
Suppose that there is no edge between $i$ and $j$ in $\sk(G)$. Since the global Markov property w.r.t.\ $G$ implies the pairwise Markov property, it holds that $\langle i,j\cd C(i,j)\rangle\in\mathcal{J}$, which implies  $i$ is not adjacent to $j$ in $\sk(\mathcal{J})$.
\end{proof}
Hence, if $\mathcal{J}$ is Markov to a graph $G$ such that $\sk(G)=\sk(\mathcal{J})$ then $G$ has the fewest number of edges among those to which $\mathcal{J}$ is Markov. We say that $\mathcal{J}$ is \emph{minimally Markov} to a graph $G$ if $\mathcal{J}$ is Markov to $G$ and $\sk(G)=\sk(\mathcal{J})$. The same can be defined for probability distributions,  and has been used in the literature under the name of \emph{(causal) minimality assumption} \citep{pea09,spioo,zha08,nea04}.  Minimally Markov independence models to a graph are important since only these can also be faithful to the graph:
\begin{proposition}\label{prop:01}
If an independence model $\mathcal{J}$ and a graph $G$ are faithful then $\mathcal{J}$ is minimally Markov to $G$.
\end{proposition}
\begin{proof}
Since $\mathcal{J}$ is Markov to $G$, we need to prove that $\sk(G)=\sk(\mathcal{J})$. By Proposition \ref{prop:0}, $\sk(\mathcal{J})$ is a subgraph of $\sk(G)$. Now, suppose that there is no edge between $i$ and $j$ in $\sk(\mathcal{J})$. By the construction of $\sk(\mathcal{J})$, it holds that $\langle i,j\cd C\rangle\in\mathcal{J}$ for some $C$. Since $\mathcal{J}$ and $G$ are faithful, $i\dse j\cd C$ in $G$. This implies that $i$ is not adjacent to $j$ in $G$.
\end{proof}
Hence, we need to discuss conditions under which $\mathcal{J}$ is minimally Markov to a graph as well as conditions for faithfulness of minimally Markov independence models and graphs. These will be presented in Section \ref{sec:charac}.
\section{Preordering in Graphs and Ordered Stabilities}\label{sec:order}
In this section, we first define the preordering for sets and its validity for nodes of the graphs, and then use preordering to define some new properties of conditional independence.
\subsection{Preordering of the Nodes of a Graph} Over a set $V$, a \emph{partial order} is a binary relation $\leq$ that satisfies the following properties for all members $a,b,c\in V$:
\begin{itemize}
  \item $a\leq a$ (\emph{reflexivity});
  \item if $a\leq b$ and $b\leq a$ then $a=b$ (\emph{antisymmetry});
  \item if $a\leq b$ and $b\leq c$ then $a\leq c$ (\emph{transitivity}).
\end{itemize}
If $a \leq b$ or $b \leq a$ then $a$ and $b$ are \emph{comparable}; otherwise they are \emph{incomparable}.  A partial order under which every pair of elements is comparable is called a \emph{total order} or a \emph{linear order}.

A \emph{preorder} $\lesssim$ over $V$, on the other hand,  is a binary relation that is only reflexive and transitive. Given a preorder $\lesssim$ over $V$, one may define an equivalence relation $\sim$ on $V$ such that $a \sim b$ if and only if $a \lesssim b$ and $b \lesssim a$. This explains the use of the notation  $\lesssim$. It then holds that
\begin{itemize}
  \item if $a\lesssim b$ and $b\sim c$ then $a\lesssim c$; and
  \item if $a\lesssim b$ and $a\sim c$ then $c\lesssim b$.
\end{itemize}
We also use the notation $a<b$ indicating that $a\lesssim b$ and $a\not\sim b$.

We first provide the following result; see Section 5.2.1 of \citet{sch03}:
\begin{proposition}\label{prop:500}
Let $\lesssim$ be a preorder over the set $V$, and $\sim$ the equivalence relation on $V$ as defined above. Let also $Q$ be the set of all equivalence classes of $V$ w.r.t.\ $\sim$. Then the relation $\leq$ defined over $Q$ by $\big[[a]\leq [b]\iff a\lesssim b\big]$ is a partial order over $Q$.
\end{proposition}

We say that a  graph $G=(V,E)$ admits a \emph{valid preorder} $\lesssim$ if, for nodes $i$ and $j$ of $G$, the following holds:
\begin{itemize}
  \item if $i\ful j$ then $i\sim j$;
  \item if $i\fra j$ then $j<i$;
  \item if $i\arc j$ then $i$ and $j$ are incomparable.
\end{itemize}
The global interpretation of a valid preorder on graphs is as follows:
\begin{proposition}\label{prop:50}
Let $\lesssim$ be a valid preorder for a graph $G$. It then holds for nodes $i$ and $j$ of $G$  that if $i\in\ant(j)$ then $j\lesssim i$; in particular,
\begin{enumerate}
  \item if there is a semi-directed path from $i$ to $j$ then $j< i$; and
  \item if $i$ and $j$ are connected by a path consisting only of lines then $i\sim j$.
\end{enumerate}
\end{proposition}
\begin{proof}
The proof follows from transitivity of both preorders and anterior paths.
\end{proof}

\begin{corollary}
Let $\lesssim$ be a valid preorder for a graph $G$ and $G_{-}$ its subgraph induced by lines. The equivalence classes of $G$ based on $\lesssim$ are the connected components of $G_{-}$. In addition, $\leq$ defines a partial order over the connected components of $G_{-}$ by  [if $i\fra j$, $i\in \tau$ and $j\in \delta$  then $\tau>\delta$], for $\tau$ and $\delta$ two connected components of $G_{-}$.
\end{corollary}
\begin{proof}
The results follows from Propositions \ref{prop:50} and \ref{prop:500}.
\end{proof}
For example, consider the graph in Fig.\ \ref{fig:ordex}. The graph admits a valid preorder with the preorder provided over the nodes, or equivalently over the connected components of $G_{-}$ (with the abuse of the notation for numbering). The notation implies two nodes with label $2$ are in the same equivalence class, $2>1$ and $2'>1'$, but $a$, $b'$, and $c''$, $a,b,c\in\{1,2\}$, are not comparable.
\begin{figure}[h]
\centering
\scalebox{0.65}{\includegraphics{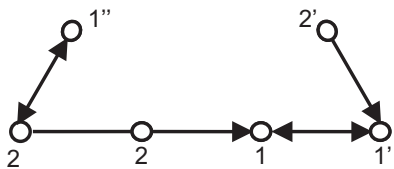}}
  \caption[]{\small{A preordered graph.}}
     \label{fig:ordex}
\end{figure}

There may be many different preorders that are valid for a graph. If there is a valid preordering and we expect the other direction of Proposition \ref{prop:50} to hold  then we obtain a unique preordering for the graph: Given a graph $G$, if $i\notin\ant(j)$ and $j\notin\ant(i)$ then set $i$ and $j$ to be incomparable. Otherwise, let $i\sim j$ when $i\ful j$, and $i>j$ when $i\fra j$. It is easy to see that this, in fact, is a preorder for the nodes of $G$. We call this preordering the \emph{minimal preorder} for $G$ since it gives the fewest possible comparable pairs of nodes. For example, the preordering in the graph of Fig.\ \ref{fig:ordex} is minimal. In this paper, we mostly deal with this type of preorderings for graphs. It is easy to observe the following:
\begin{proposition}\label{prop:600}
If $G$ is anterial then the minimal preorder for $G$ is a valid preorder for $G$.
\end{proposition}
In fact, in general, we have the following:
\begin{proposition}\label{prop:60}
A graph admits a valid preorder if and only if it is anterial.
\end{proposition}
\begin{proof}
If a graph $G$ admits a valid preorder then by Proposition \ref{prop:50}, there cannot be a directed cycle, nor can there be an arc with one endpoint anterior of the other. The converse is Proposition \ref{prop:600}.
\end{proof}
Therefore,  in addition to CMGs, SGs (and ADMGs) do not admit a valid preorder, but AGs do. However, notice that for every CMG, there exists a Markov equivalent AnG (see \citealt{sad16}); and the same for SGs and AG. Hence, the question of whether there is a CMG that is faithful to $\mathcal{J}$ is the same as  whether there is an AnG that is faithful to $\mathcal{J}$; and the same holds for SG and AG.

As discussed above, there is a one-to-one correspondence between AnGs and the minimal preorder for the AnGs. When the skeleton of the concerned graph is known, we have the following trivial result:
\begin{proposition}\label{prop:00}
Given the skeleton of a graph $G$ and a preorder $\lesssim $ of its nodes, $G$ can be constructed uniquely by directing the edges of $\sk(G)$. In addition, $\lesssim$ is a valid preorder for $G$.
\end{proposition}
Let $\mathcal{J}$ be defined over $V$. We introduce preordering in this paper in order to stress that, in principle, a preordering of members of $V$ (which are possibly random variables) is defined irrespective of graphs to which they are Markov or faithful. However, for our purposes, knowing the skeleton of such graphs, one can directly discover the directions of the edges of the graph, which correspond to the minimal preordering for the graph, rather than working directly on the set $V$.

Suppose that there exists a preorder $\lesssim$ over $V$. Proposition \ref{prop:00} implies that we can direct the edges of $\sk(\mathcal{J})$ based on $\lesssim$ in order to define a \emph{dependence graph} $G(\mathcal{J},\lesssim)$ induced by $\mathcal{J}$ and $\lesssim$. Notice that, by Proposition \ref{prop:60}, $G(\mathcal{J},\lesssim)$ is anterial. Similarly, for $P$, one can define $G(P,\lesssim)=G(\mathcal{J}(P),\lesssim)$.

As mentioned, we are only interested in preorderings that are minimal preorders of graphs. Given $\mathcal{J}$, we define a preordering $\lesssim$ to be \emph{$\mathcal{J}$-compatible} if $\lesssim$ is the minimal preorder for $G(\mathcal{J},\lesssim)$. Similarly, one can define a $P$-compatible preorder.

\subsection{Ordered Upward- and Downward-stabilities} We now exploit the preodering for independence models in order to define two other properties in addition to the seven properties defined in Section \ref{sec:propind} (namely singleton-transitive compositional graphoid axioms). We say that an independence model $\mathcal{J}$ over the set $V$ respectively satisfies \emph{ordered upward-} and \emph{downward-stability} w.r.t.\ a preorder $\lesssim$ of $V$ if the following hold:
  \begin{itemize}
	\item[8.] if $\langle i,j\cd C\rangle\in \mathcal{J}$  then $\langle i,j\cd C\cup\{k\}\rangle\in \mathcal{J}$ for every $k\in V\setminus\{i,j\}$ such that $l\lesssim k$ for some $l\in \{i,j\}$ or $l\sim k$ for some $l\in C$ (ordered upward-stability);
	\item[9.] if $\langle i,j\cd C\rangle\in \mathcal{J}$ then $\langle i,j\cd C\setminus\{k\}\rangle\in \mathcal{J}$ for every $k\in V\setminus\{i,j\}$ such that $l\not\lesssim k$ for every $l\in \{i,j\}$ and $l\not< k$ for every $l\in C\setminus\{k\}$ (ordered downward-stability).
\end{itemize}
Ordered upward-stability is a generalization of a modification of \emph{upward stability}, defined in \citet{fal15}, and \emph{strong union}, defined in \citet{pea85}, for undirected graphs, with singletons instead of node subsets.
\begin{proposition}\label{prop:3}
Let $G$ be an AnG and $\lesssim$ the minimal preorder for $G$. Then $\mathcal{J}(G)$ satisfies ordered upward- and downward-stability w.r.t.\ $\lesssim$.
\end{proposition}
\begin{proof}
First we prove that $\mathcal{J}(G)$ satisfies ordered upward-stability w.r.t.\ the minimal preorder: If $k\in C$, the result is obvious, thus suppose that $k\notin C$. Suppose that there is a connecting walk $\pi$ between $i$ and $j$ given $C\cup \{k\}$. If $k$ is not on $\pi$ then we are done; otherwise, $k$ is on a collider section with no other node in $C$ on $\pi$. Notice that $l\lesssim k$ means here that $k\in\ant(l)$, and $l\sim k$ that $k$ and $l$ are connected by a path consisting of lines.

If $l\in C$ then denote the path of lines between $k$ and $l$ by $\omega$. The walk consisting of the subwalk of $\pi$ from $i$ to $k$, $\omega$, $\omega$ in the reverse direction (from $l$ to $k$), and the subwalk of $\pi$ from $k$ to $j$ is connecting given $C$. This is because on this walk 
$k$ and $l$ are in the same collider section with $l$ in $C$.

Now denote an anterior path from $k$ to $l\in \{i,j\}$ by $\varpi$. Without loss of generality we can assume that $l=j$. We now consider two cases: 1) If $\varpi$ has no nodes in $C$ then the walk containing the subwalk of $\pi$ between $i$ and $k$ and $\varpi$ is connecting given $C$ between $i$ and $j$ since $k$ and all nodes on $\varpi$ are on non-collider sections on this walk. 2) If $\varpi$ has a node in $C$ then consider the one closest to $k$ and call it $h$. Then the walk consisting of the subwalk of $\pi$ from $i$ to $k$, the subpath of $\varpi$ from $k$ to $h$, the reverse of the subpath from $h$ to $k$, and the subwalk of $\pi$ from $k$ to $j$ is connecting given $C$ as proven as follows: 1) if there is an arrow on $\varpi$ then $k$ is on a non-collider section in both instances in which it appears on this walk, and $h$ is on a collider section and in $C$. 2) if $\varpi$ only consists of lines then $k$ and $h$ are in the same collider section with a node ($h$) in $C$.

Now we prove that $\mathcal{J}(G)$ satisfies ordered downward-stability: If $k\notin C$ then the result is obvious, thus suppose that $k\in C$. Suppose that there is a connecting walk $\pi$ between $i$ and $j$ given $C\setminus\{k\}$. If $k$ is not on $\pi$ then we are done. Otherwise, first it is clear that since $\pi$ is connecting given $C\setminus\{k\}$, $k$ cannot be the only node on a collider section that is in $C$.

If $k$ is on the same section as another member of $C$ then $\pi$ is clearly connecting given $C$. The only case that is left is when $k$ is on a non-collider section on $\pi$; but this is impossible: there is no arrowhead at the section containing $k$ from one side on $\pi$, say from the
$j$ side. By moving towards $j$ from $k$, it can be implied that $k$ is either an anterior of a node in a collider section that is in $C$, or an anterior of $j$, but both are impossible since the former implies $l< k$ for $l\in C$, and the latter implies $j\lesssim k$.
\end{proof}

\section{Characterization of Graphical Independence Models}\label{sec:charac}
In this section, we first provide sufficient conditions for minimal Markovness, and then necessary and sufficient conditions for faithfulness.
\subsection{Sufficient Conditions for Minimal Markovness}\label{sec:mark}
Here we provide sufficient conditions for an independence model to be minimally Markov to a graph. First, we have the following trivial result:
\begin{proposition}\label{prop:minord}
If $\mathcal{J}$ is minimally Markov to an anterial graph $G$ then $G=G(\mathcal{J},\lesssim)$ for the minimal preorder $\lesssim$ for $G$.
\end{proposition}
By Proposition \ref{prop:01}, the above statement also holds when $\mathcal{J}$ and $G$ are faithful.
\begin{proposition}\label{prop:80}
Suppose that there exists a $\mathcal{J}$-compatible preorder $\lesssim$ over $V$ w.r.t.\ which $\mathcal{J}$ satisfies ordered downward- and upward-stability. It then holds that $\mathcal{J}$ is pairwise Markov to $G(\mathcal{J},\lesssim)$.
\end{proposition}
\begin{proof}
Suppose that there are non-adjacent nodes $i$ and $j$ in $G(\mathcal{J},\lesssim)$. These are non-adjacent in $\sk(\mathcal{J})$ too. By definition of $\sk(\mathcal{J})$, it holds that $\langle i,j\cd C\rangle\in \mathcal{J}$ for some $C\subseteq V\setminus\{i,j\}$.

Since $\mathcal{J}$ satisfies ordered downward-stability and since the preorder is $\mathcal{J}$-compatible, as long as a node $k\in C$ is not an anterior of $\{i,j\}$ or  there is no semi-directed path from it to $C\setminus\{k\}$, it can be removed from the conditioning set, that is we have that $\langle i,j\cd C\setminus k\rangle\in \mathcal{J}$. Since $G(\mathcal{J},\lesssim)$ is acyclic, as long as $\ant(\{i,j\})\subset C$, such a $k$ exists.  Now consider a node $k'\in C\setminus\{k\}$ that is not an anterior of $\{i,j\}$ or there is no semi-directed path from it to $C\setminus\{k,k'\}$ and apply downward-stability again. By an inductive argument we imply  that $\langle i,j\cd C\cap\ant(\{i,j\})\rangle\in \mathcal{J}$.
Now since $P$ satisfies ordered upward-stability, nodes outside $C$ that are in $\ant(\{i,j\})$ can be added to the conditioning set. Hence, it holds that $\langle i,j\cd \ant(\{i,j\})\setminus\{i,j\}\rangle\in \mathcal{J}$. This completes the proof.
\end{proof}
In principle, a compatible preorder w.r.t.\ which $\mathcal{J}$ satisfies ordered downward- and upward-stability can be found by going through all such preorderings. Finding such a preorder efficiently is a structural learning task. We believe we have an efficient way to do this when the skeleton of $\mathcal{J}$ has been learned, but this is beyond the scope of this paper.
\begin{theorem}\label{thm:0}
Suppose, for an independence model $\mathcal{J}$ over $V$, that
\begin{enumerate}
  \item $\mathcal{J}$ is a compositional graphoid; and
  \item  there exists a $\mathcal{J}$-compatible preorder $\lesssim$ over $V$ w.r.t.\ which $\mathcal{J}$ satisfies ordered downward- and upward-stability.
\end{enumerate}
It then holds that $\mathcal{J}$ is minimally Markov to $G(\mathcal{J},\lesssim)$.
\end{theorem}
\begin{proof}
The proof follows from Proposition \ref{prop:80}, the fact that $\sk(\mathcal{J})=\sk(G(\mathcal{J},\lesssim))$, and Proposition \ref{prop:114}.
\end{proof}

\subsection{Conditions for Faithfulness}\label{sec:faith}
In this section, we present the main result of this paper, which provides necessary and sufficient conditions for faithfulness of independence models (and probability distributions) and a (not necessarily known) graph in the general case. In Section \ref{sec:spec}, we specialize the result to the more well-known subclasses as corollaries.

\begin{proposition}\label{prop:10}
Let $\mathcal{J}$ be an independence model over $V$. Suppose that $\mathcal{J}$ is  minimally Markov to an anterial graph $G$. It then holds that $\mathcal{J}$ and $G$ are faithful if and only if
\begin{enumerate}
  \item $\mathcal{J}$ satisfies singleton-transitivity; and
  \item $\mathcal{J}$ satisfies ordered downward- and upward-stability w.r.t.\ the minimal preorder for $G$.
\end{enumerate}
\end{proposition}
\begin{proof}
Suppose that $\mathcal{J}$ and $G$ are faithful. This is equivalent to $\mathcal{J}=\mathcal{J}(G)$. By Proposition \ref{prop:1}, $\mathcal{J}$ satisfies singleton-transitivity, and by Proposition \ref{prop:3}, the result follows.

Conversely, suppose that $\mathcal{J}$ satisfies singleton-transitivity as well as ordered downward- and upward-stability w.r.t.\ the minimal preorder $\lesssim$ for $G$. By Proposition \ref{prop:minord}, $G=G(\mathcal{J},\lesssim)$. We show that $\mathcal{J}$ and $G$ are faithful.  We need to show that if $\langle A,B\cd C\rangle\in \mathcal{J}$ then $A\dse B\cd C$ in $G$. Consider $i\in A$ and $j\in B$. By decomposition, we have that $\langle i,j\cd C\rangle\in \mathcal{J}$, which implies that $i$ and $j$ are not adjacent in $G$.

If, for contradiction, $C$ does not separate nodes $i$ and $j$ then there exists a connecting walk $\pi=\langle i=i_1, i_2,\dots , i_r = j\rangle$, on which all collider sections have a node in $C$ and all non-collider sections are outside $C$. 
In addition, for every node $i_q$, $2\leq q\leq r-1$, on $\pi$ define $C_{i_q}$ as follows: If $i_q\notin C$ then $C_{i_q}=\varnothing$; if $i_q\in C$ then $C_{i_q}$ is the set of all $l\in C\setminus \pi$ such that there is a semi-directed path from $i_q$ to $l$ but not from any $i_p\in C$ to $l$, $p\in\{q+1,\dots,r-1\}$. Let $C_{\pi}=\bigcup_{q=2}^{r-1}C_{i_q}$.

We prove by induction along nodes of $\pi$ that for every $q$, $2\leq q\leq r$, it holds that $\langle i_1,i_q\cd (\cup_{p=2}^{q-1} C_{i_p})\cup [C\setminus(\{i_q,\dots,i_{r}\}\cup C_{\pi})]\rangle\notin \mathcal{J}$. The base for $q=2$ is trivial since by the construction of $G(\mathcal{J},\lesssim)$, adjacent nodes are dependent given anything. Suppose now that the result holds for $q$ and we prove it for $q+1$.

First, again by the construction of $G(\mathcal{J},\lesssim)$, it holds that $\langle i_q,i_{q+1}\cd (\cup_{p=2}^{q-1} C_{i_p})\cup [C\setminus(\{i_q,\dots,i_r\}\cup C_{\pi})]\rangle\notin \mathcal{J}$.  By singleton-transitivity, this and the induction hypothesis imply (a): $\langle i_1,i_{q+1}\cd (\cup_{p=2}^{q-1} C_{i_p})\cup [C\setminus(\{i_q,\dots,i_r\}\cup C_{\pi})]\rangle\notin \mathcal{J}$; or (b): $\langle i_1,i_{q+1}\cd \{i_q\}\cup (\cup_{p=2}^{q-1} C_{i_p})\cup [C\setminus(\{i_q,\dots,i_r\}\cup C_{\pi})]\rangle\notin \mathcal{J}$. We now have two cases:

1) If $i_q$ is on a collider section: The section containing $i_q$ has a node in $C$. We again consider two subcases:

1.i) If $i_q$ is in $C$:  If (a) holds then first notice that $i_q$ is not an anterior of \{$i_1,i_{q+1}\}$ nor is there a semi-directed path from $i_q$ to $(\cup_{p=2}^{q-1} C_{i_p})\cup [C\setminus(\{i_q,\dots,i_r\}\cup C_{\pi})]$.  We apply downward-stability to (a) to obtain (b). If $C_{i_q}\neq\varnothing$ then consider a highest order node $l$ in  $C_{i_q}$. Again by downward-stability, we obtain  $\langle i_1,i_{q+1}\cd \{l\} \cup (\cup_{p=2}^{q-1} C_{i_p})\cup [C\setminus(\{i_{q+1},\dots,i_r\}\cup C_{\pi})]\rangle\notin \mathcal{J}$. By an inductive argument along the members of $C_{i_q}$ (by, at each step, choosing a highest order node in $C_{i_q}$ that has not yet been chosen), we eventually obtain $\langle i_1,i_{q+1}\cd (\cup_{p=2}^q C_{i_p})\cup [C\setminus(\{i_{q+1},\dots,i_r\}\cup C_{\pi})]\rangle\notin \mathcal{J}$. Notice that $i_q$ is in the conditioning set of this dependency.

1.ii) If $i_q$ is not in $C$:  If (b) holds then observe that $i_q$ is in the same equivalence class as that of either $i_{q+1}$ or a node $i_s\in C$, $s<q$. 
Hence we can apply upward-stability to obtain  (a). Since $C_{i_2}=\varnothing$, (a) is clearly the same as $\langle i_1,i_{q+1}\cd (\cup_{p=2}^q C_{i_p})\cup [C\setminus(\{i_{q+1},\dots,i_r\}\cup C_{\pi})]\rangle\notin \mathcal{J}$. Notice that $i_q$ is not in the conditioning set of this dependency.

2) If $i_q$ is on a non-collider section: We have that $i_q\notin C$, and $i_q\in\ant(\{i_1,i_{q+1}\})$. Hence, by upward-stability, from (b) we obtain (a). Again, since $C_{i_q}=\varnothing$, we obtain the result.

Therefore, we established that $\langle i_1,i_q\cd (\cup_{p=2}^{q-1} C_{i_p})\cup [C\setminus(\{i_q,\dots,i_{r}\}\cup C_{\pi})]\rangle\notin \mathcal{J}$, for all $q$, where depending on whether $i_q\in C$ or not, the conditioning set contains or does not contain $i_q$. By letting $q=r$, we obtain $\langle i_1,i_r\cd C\rangle\notin \mathcal{J}$, a contradiction. Therefore, it holds that $i\dse j \cd C$ in $G$. By the composition property for separation for graphs (Proposition \ref{prop:110}), we obtain the result.
\end{proof}
We follow the arguments in the proof for the following example.
\begin{example}
Consider the  undirected (LWF) chain graph $G$ in Fig.\ \ref{fig:ex1}. Suppose that a probability distribution $P$ is minimally Markov to $G$, and $\mathcal{J}(P)$ satisfies singleton-transitivity and ordered downward- and upward-stability w.r.t.\ the minimal preorder for $G$. In order to prove faithfulness, one needs to show that if two sets of nodes are not separated given a third set then they are not independent either.

For example, we have that $l\notdse j\cd l_2$. Notice that this is a case where there is no path that connects  $l$ and $j$. In order to show that $l\notci j\cd l_2$, (as in the proof of the above theorem) we start from $l=l_1$ and move towards $j=l_7$ via the connecting walk $\langle l_1,l_2,l_3,l_4,l_5,l_6,l_7 \rangle$ given $l_2$.
\begin{figure}[h]
\centering
\scalebox{0.28}{\includegraphics{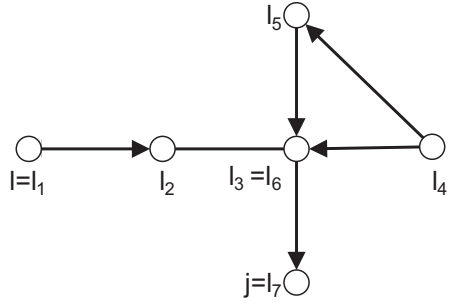}}
  \caption[]{\small{An undirected (LWF) chain graph.}}
     \label{fig:ex1}
\end{figure}

We abbreviate singleton-transitivity, ordered upward-stability, and ordered downward stability by ``$st$'', ``$ous$'', and ``$ods$'' respectively. We write the arguments in a condensed from. For example, the first implication (on the left hand side) says that $l_1\notci l_2$  and $l_2\notci l_3$, using singleton-transitivity, imply that either $l_1\notci l_3$ or $l_1\notci l_3\cd l_2$ ; but now ordered downward-stability implies that we must have $l_1\notci l_3\cd l_2$.\\

\begin{tabular}{@{}c@{}@{}c@{}@{}c@{}@{}c@{}@{}@{}c@{}@{}c@{}@{}c@{}@{}c@{}@{}c@{}@{}c@{}@{}@{}c@{}}
$l_1\notci l_2$ & & $l_1\notci l_3$ & & $l_4\notci l_5\cd l_2$ & &  $l_1\notci l_5\cd \{l_2,l_4\}$ & & $l_6\notci l_7\cd l_2$ &&  $l_1\notci l_7\cd \{l_2,l_6\}$  \nls
	       &$\stackrel{st}{\Rightarrow}$ & \small{or \nn\n $\Downarrow$\hspace{-1mm} ods}&&&  \hspace{-2mm}$\stackrel{st}{\Rightarrow}$ & \small{or \nn\n $\Downarrow$\hspace{-1mm} ous} & & & \hspace{-2mm}$\stackrel{st}{\Rightarrow}$ &  \small{or \nn\n $\Downarrow$\hspace{-1mm} ous} \nls
$l_2\notci l_3$ & & $l_1\notci l_3\cd l_2$ & & $l_1\notci l_4\cd l_2$ & &  $l_1\notci l_5\cd l_2$ & & $l_1\notci l_6\cd l_2$ & &   $l_1\notci l_7\cd l_2$.\nls
&&&$\stackrel{st}{\Rightarrow}$& \small{or \nn\n $\Uparrow$\hspace{-1mm} ous} & & &   \hspace{-2mm}$\stackrel{st}{\Rightarrow}$ & \small{or \nn\n $\Uparrow$\hspace{-1mm} ous} & &  \nls
&&$l_3\notci l_4\cd l_2$ & & $l_1\notci l_3\cd \{l_2,l_3\}$ & & $l_5\notci l_6\cd l_2$ & & $l_1\notci l_6\cd \{l_2,l_5\}$ &&
\end{tabular}
\end{example}
The following example shows how singleton-transitivity and ordered stabilities are necessary for faithfulness.
\begin{example}
Consider the DAG $G$ in Fig.\ \ref{fig:ex2}. If an independence model $\mathcal{J}$ is minimally Markov to $G$ then we must have $\langle h,k\cd \{j,l\}\rangle\in \mathcal{J}$  and $\langle j,l\cd k\rangle\in \mathcal{J}$. Now, in violation of faithfulness, if only $\langle h,k\cd \varnothing\rangle\in \mathcal{J}$ in addition to these holds then upward-stability is violated since for $j>h$, $\langle h,k\cd j\rangle\in \mathcal{J}$  must hold.   If, in violation of faithfulness, only $\langle h,k\cd j\rangle\in \mathcal{J}$ in addition to the original statements holds then singleton-transitivity is violated since  $\langle h,k\cd j\rangle\in \mathcal{J}$ and $\langle h,k\cd \{j,l\}\rangle\in \mathcal{J}$ must imply either $\langle h,l\cd j\rangle\in \mathcal{J}$  or $\langle k,l\cd j\rangle\in \mathcal{J}$, which are both impossible due to minimality since $h,l$ and $k,l$ are adjacent in $G$.
\begin{figure}[h]
\centering
\scalebox{0.28}{\includegraphics{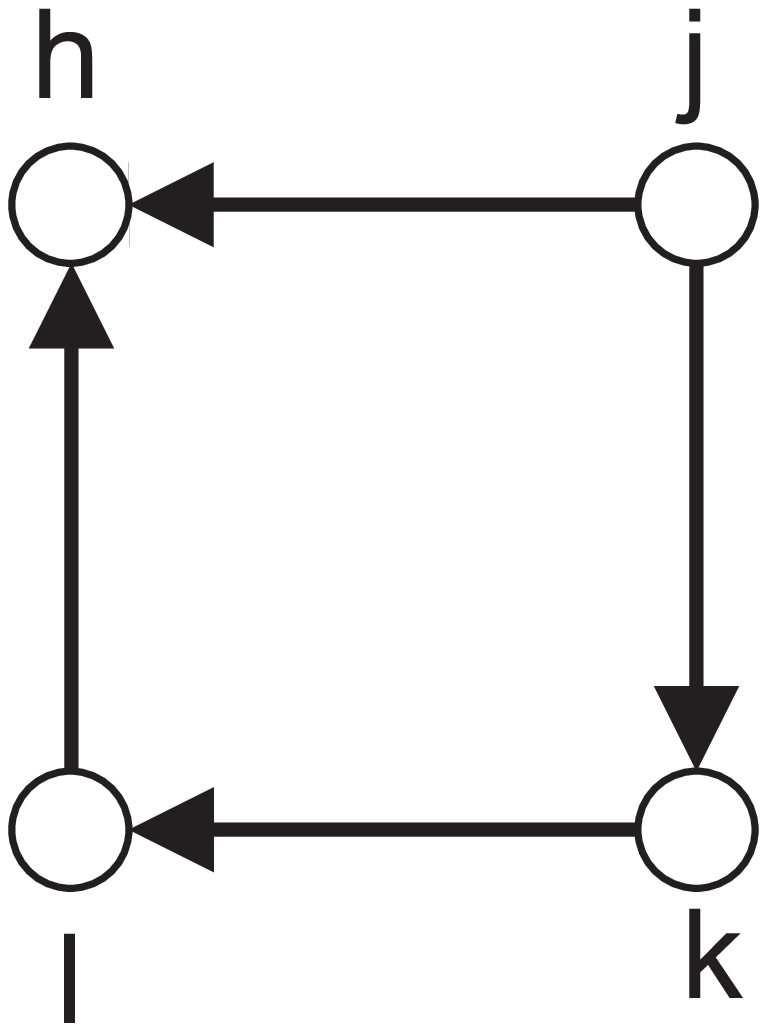}}
  \caption[]{\small{A DAG.}}
     \label{fig:ex2}
\end{figure}
\end{example}

Without assuming the Markov assumption we have the following:
\begin{theorem}\label{thm:2}
Let $\mathcal{J}$ be an independence model defined over $V$. It then holds that $\mathcal{J}$ is graphical if and only if
\begin{enumerate}
  \item $\mathcal{J}$ is a singleton-transitive compositional graphoid; and
  \item there exists a $\mathcal{J}$-compatible preorder $\lesssim$ over $V$ w.r.t.\ which $\mathcal{J}$ satisfies ordered downward- and upward-stability.
\end{enumerate}
In addition, if 1 and 2 hold then $\mathcal{J}$ is faithful to $G(\mathcal{J},\lesssim)$.
\end{theorem}
\begin{proof}
($\Rightarrow$) follows from Propositions \ref{prop:110}, \ref{prop:1}, and \ref{prop:3}, and the fact that there is always an AnG that is Markov equivalent to a given CMG. To prove ($\Leftarrow$) and the last statement of the theorem, consider $G(\mathcal{J},\lesssim)$. Theorem \ref{thm:0} implies that $\mathcal{J}$ is minimally Markov to $G(\mathcal{J},\lesssim)$. This together with the $\mathcal{J}$-compatibility of the preorder, using Proposition \ref{prop:10}, implies that $\mathcal{J}$ and $G(\mathcal{J},\lesssim)$ are faithful.
\end{proof}
Hence, for probability distributions, we have the following characterization:
\begin{corollary}\label{coro:vvn}
Let $P$ be a probability distribution defined over $\{X_\alpha\}_{\alpha\in V}$. It then holds that $P$ is graphical if and only if
\begin{enumerate}
  \item $\mathcal{J}(P)$ satisfies intersection, composition, and singleton-transitivity; and
  \item there exists a $P$-compatible preorder $\lesssim$ over $V$ w.r.t.\ which $\mathcal{J}(P)$ satisfies ordered downward- and upward-stability.
\end{enumerate}
In addition, if 1 and 2 hold then $P$ is faithful to $G(P,\lesssim)$.
\end{corollary}

In fact, we have shown in Theorem \ref{thm:2} that if $\mathcal{J}$ is graphical then for every $\mathcal{J}$-compatible preorder w.r.t.\ which $\mathcal{J}$ satisfies ordered downward- and upward-stability, $G(\mathcal{J},\lesssim)$ is a graph that is faithful to $\mathcal{J}$. A consequence of Proposition \ref{prop:minord} implies that such graphs (that is $G=G(\mathcal{J},\lesssim)$ for some preorder $\lesssim$ w.r.t.\ which $\mathcal{J}$ satisfies ordered downward- and upward-stability) constitute the set of all graphs that are faithful to $\mathcal{J}$, that is a Markov equivalence class of graphs whose members are faithful to $\mathcal{J}$. Every member of the equivalence class corresponds to a different preorder  w.r.t.\ which $\mathcal{J}$ satisfies ordered downward- and upward-stability. This equivalence set can be represented by one member (a graph with the least number of arrowheads) similar to the idea of a CPDAG for the Markov equivalence class of DAGs (see \citealt{spioo}), but we do not go through the details of this here.

Of course, if  the goal is to verify whether $\mathcal{J}$ and a given $G$ are faithful then we can use the following corollary:
\begin{corollary}\label{coro:2}
Let $\mathcal{J}$ be an independence model defined over $V$, and $G$ an AnG with node set $V$. It then holds that if
\begin{enumerate}
  \item $\mathcal{J}$ is a singleton-transitive compositional graphoid; and
  \item $\mathcal{J}$ satisfies ordered downward- and upward-stability w.r.t.\ the minimal preorder for $G$,
\end{enumerate}
then $\mathcal{J}$ and $G$ are faithful.
\end{corollary}

It is also important to note  that singleton-transitivity and ordered upward or downward-stability are necessary in the sense that they are not implied by one another, although there are different ways to axiomatize singleton-transitive compositional graphoids that satisfy ordered upward- and downward-stability.

If $P$ is Gaussian then we have the following:
\begin{corollary}\label{coro:gaus}
Let $P$ be regular Gaussian distribution. It then holds that $P$ is graphical if and only if there exists a $P$-compatible preorder $\lesssim$ w.r.t.\ which $\mathcal{J}(P)$ satisfies ordered downward- and upward-stability.
\end{corollary}
\begin{proof}
The proof follows from the fact that regular Gaussian distributions satisfy intersection, composition, and singleton-transitivity.
\end{proof}

\section{Specialization to Subclasses}\label{sec:spec}
The definitions and results presented in this paper can be specialized to any subclass of CMGs including ancestral graphs and LWF chain graphs. However, here we only focus on the three of the most used classes of undirected (concentration), bidirected (covariance), and directed acyclic graphs.
\subsection{Specialization to Undirected and Bidirected Graphs}\label{sec:un-bi}
For connected UGs, every two nodes are in the same equivalence class, and for BGs every two nodes are incomparable; hence, the minimal (pre)ordering is trivial and uninteresting in these cases. More precisely, ordered upward- and downward-stabilities can be specialized for these two types of  trivial preordering:
\begin{itemize}
	\item[8a.] if $\langle i,j\cd C\rangle\in \mathcal{J}$ then $\langle i,j\cd C\cup\{k\}\rangle\in \mathcal{J}$ for every $k\in V\setminus\{i,j\}$ (\emph{upward-stability});
	\item[9a.] if $\langle i,j\cd C\rangle\in \mathcal{J}$ then $\langle i,j\cd C\setminus\{k\}\rangle\in \mathcal{J}$ for every $k\in V\setminus\{i,j\}$ (\emph{downward-stability}).
\end{itemize}
\begin{proposition}\label{prop:m}
An independence model $\mathcal{J}$ satisfies ordered upward- and downward-stability (8) and (9) w.r.t.\ the minimal preorder if and only if,
\begin{enumerate}
  \item for connected undirected graphs, it satisfies upward-stability (8a);
  \item for bidirected graphs, it satisfies downward-stability (9a).
\end{enumerate}
In addition, if, for undirected graphs, $\mathcal{J}$ satisfies (8a) then it satisfies (8).
\end{proposition}
\begin{proof}
The proof of 1 and 2 are straightforward since in 1 all nodes are in the same equivalence class, and in 2 all nodes are incomparable. The second part is trivial.
%
\end{proof}
Analogous to Proposition \ref{prop:3}, it is also straightforward to show the following:
\begin{proposition}\label{prop:2}
It holds that
\begin{enumerate}
  \item for an undirected graph $G$, $\mathcal{J}(G)$ satisfies upward-stability;
  \item for a bidirected graph $G$, $\mathcal{J}(G)$ satisfies downward-stability.
\end{enumerate}
\end{proposition}
For UGs and BGs, denote the  unique and trivial valid preorders by $\lesssim_u^*$ and $\lesssim_b^*$ respectively. For UGs, $G^*_u(\mathcal{J}):=G(\mathcal{J},\lesssim_u^*)=\sk(\mathcal{J})$, whereas for BGs, $G^*_b(\mathcal{J}):=G(\mathcal{J},\lesssim_b^*)$ is $\sk(\mathcal{J})$ with all edges directed to be bidirected. Another way to construct induced UGs and BGs by $\mathcal{J}$ is to construct them based on their corresponding  pairwise Markov property, that is to connect every pair of nodes $i$ and $j$ if $\langle i,j\cd V\setminus\{i,j\}\rangle\notin \mathcal{J}$ for UGs, and if $\langle i, j\cd \varnothing\rangle\notin \mathcal{J}$ for BGs. Denote these graphs by $G_u(\mathcal{J})$ and $G_b(\mathcal{J})$ respectively. We then have the following:
\begin{proposition}\label{coro:4}
It holds that,
\begin{enumerate}
  \item for undirected graphs, $G^*_u(\mathcal{J})=G_u(\mathcal{J})$ if $\mathcal{J}$ satisfies upward-stability;
  \item for bidirected graphs, $G^*_b(\mathcal{J})=G_b(\mathcal{J})$ if $\mathcal{J}$ satisfies downward-stability.
\end{enumerate}
\end{proposition}

For upward- and downward-stability, we also have the following observation:
\begin{lemma}\label{lem:1}
For an independence model $\mathcal{J}$ the following holds:
\begin{enumerate}
  \item If $\mathcal{J}$ is a semi-graphoid and satisfies upward-stability then $\mathcal{J}$ satisfies composition.
  \item If $\mathcal{J}$ is a semi-graphoid and satisfies downward-stability then $\mathcal{J}$ satisfies intersection.
\end{enumerate}
\end{lemma}
\begin{proof}
1.\ Suppose that $\langle A,B\cd C\rangle\in \mathcal{J}$ and $\langle A,D\cd C\rangle\in \mathcal{J}$. By decomposition, it holds, for $i\in A$, $j\in B$, and $k\in D$, that $\langle i, j\cd C\rangle\in \mathcal{J}$ and $\langle i, k\cd C\rangle\in \mathcal{J}$. By upward-stability, we obtain $\langle i, j\cd C\cup \{k\}\rangle\in \mathcal{J}$. Contraction implies $\langle i,\{j,k\}\cd C\rangle\in \mathcal{J}$. By an inductive argument, we obtain the result.

2.\ Suppose that $\langle A, B\cd C\cup D\rangle\in \mathcal{J}$ and $\langle A, D\cd C\cup B\rangle\in \mathcal{J}$.  By decomposition, it holds, for $i\in A$, $j\in B$, and $k\in D$, that $\langle i, j\cd C\cup \{k\}\rangle\in \mathcal{J}$ and $\langle i, k\cd C\cup \{j\}\rangle\in \mathcal{J}$. By downward-stability, we obtain $\langle i, k\cd C\rangle\in \mathcal{J}$. Contraction implies $\langle i, \{j,k\}\cd C\rangle\in \mathcal{J}$.  By an inductive argument, we obtain the result.
\end{proof}

\begin{lemma}\label{lem:2}
Let $\mathcal{J}$ be an independence model, and suppose that $\mathcal{J}$ is a graphoid. It then holds that
$\mathcal{J}$ satisfies ordered downward-stability (9) w.r.t.\ the minimal preorder for $G_u(\mathcal{J})$.
\end{lemma}
\begin{proof}
 By definition, we know that $\mathcal{J}$ is pairwise Markov to $G_u(\mathcal{J})$. It is known that under graphoid, this implies that $\mathcal{J}$ is Markov to $G_u(\mathcal{J})$; see \citet{pea88}. Now suppose that $\langle i, j\cd C\rangle\in \mathcal{J}$. The only applicable $k$ in the definition of (9) is not in the same connected component as that of $i$ or $j$. Hence, since $i\dse j\cd C\setminus\{k\}$, by the global Markov property, it holds that $\langle i, j\cd C\setminus\{k\}\rangle\in \mathcal{J}$.
\end{proof}
As a consequence of Theorem \ref{thm:0}, we have the following for UGs and BGs:
\begin{corollary}\label{coro:vvvn}
It holds that
\begin{enumerate}
  \item if $\mathcal{J}$ is a graphoid and satisfies upward-stability then $\mathcal{J}$ is minimally Markov to the undirected graph $G_u(\mathcal{J})$;
  \item if $\mathcal{J}$ is a compositional semi-graphoid and  satisfies downward-stability then $\mathcal{J}$ is minimally Markov to the bidirected graph $G_b(\mathcal{J})$.
\end{enumerate}
\end{corollary}
\begin{proof}
The proof follows from Propositions  \ref{coro:4} and \ref{prop:m}, and Lemmas \ref{lem:1} and \ref{lem:2}.
\end{proof}
For other classes of graphs, the preordering (or ordering) of nodes (and consequently random variables) is necessary since the conditioning set $C(i,j)$ for the pairwise Markov property depends on the preordering of the variables. This implies that in general one cannot construct more than the skeleton of graphs based on the pairwise Markov property if there is no given preordering.

Proposition \ref{prop:10} specializes to the following two corollaries:
\begin{corollary}
Suppose that an independence model $\mathcal{J}$ is  Markov to $G_u(\mathcal{J})$. Then $\mathcal{J}$ and $G_u(\mathcal{J})$ are faithful if and only if $\mathcal{J}$ satisfies singleton-transitivity and upward-stability.
\end{corollary}
\begin{proof}
The proof follows from Propositions \ref{prop:m}, \ref{prop:2}, and \ref{coro:4}, and Lemma \ref{lem:2}.
\end{proof}
Similarly, for BGs, we have the following:
\begin{corollary}
Suppose that an independence model $\mathcal{J}$ is  Markov to $G_b(\mathcal{J})$. Then $\mathcal{J}$ and $G_b(\mathcal{J})$ are faithful if and only if $\mathcal{J}$ satisfies singleton-transitivity and downward-stability.
\end{corollary}
The main conditions for an independence model to be faithful to a UG or BG (consequence of Theorem \ref{thm:2}) are as follows; see also Theorem 3 of \cite{pea85} for a similar result for UGs:
\begin{corollary}\label{coro:und}
Let $\mathcal{J}$ be an independence model defined over $V$. It then holds that $\mathcal{J}$ is faithful to an undirected graph if and only if it is an upward-stable singleton-transitive graphoid. In addition, if the conditions hold then $\mathcal{J}$ is faithful to $G_u(\mathcal{J})$.
\end{corollary}
\begin{proof}
The proof follows from Propositions \ref{prop:m}, \ref{prop:2}, and \ref{coro:4}, and Lemmas \ref{lem:1} and \ref{lem:2}.
\end{proof}
\begin{corollary}\label{coro:bid}
Let $\mathcal{J}$ be an independence model defined over $V$. It then holds that $\mathcal{J}$ is faithful to a bidirected graph  if and only if it is a downward-stable singleton-transitive compositional semi-graphoid. In addition, if the conditions hold then $\mathcal{J}$ is faithful to $G_b(\mathcal{J})$.
\end{corollary}
Therefore, a probability distribution is faithful to an undirected graph if and only if it satisfies intersection, singleton-transitivity, and upward-stability; and to a bidirected graph if and only if it satisfies composition, singleton-transitivity, and downward-stability.
 
 Finally, as a consequence of Corollary \ref{coro:gaus}, we have the following:
\begin{corollary}\label{coro:undgau}
Let $P$ be a regular Gaussian distribution. It then holds that $P$ being faithful to an undirected graph is equivalent to it satisfying upward-stability; and $P$ being faithful to a bidirected graph is equivalent to it satisfying downward-stability.
\end{corollary}
\subsection{Specialization to Directed Acyclic Graphs}
For DAGs, the valid preorder specializes to a valid (partial) order since there are no lines in the graph: If $i\fra j$ then $i>j$. In the minimal order for a DAG, two nodes are incomparable if and only if neither is an ancestor of the other.
Therefore, we define ordered stabilities  w.r.t.\ ordering $\leq$ for DAGs:
  \begin{itemize}
	\item[8b.] if $\langle i,j\cd C\rangle\in \mathcal{J}$  then $\langle i,j\cd C\cup\{k\}\rangle\in \mathcal{J}$ for every $k\in V\setminus \{i,j\}$ such that $l< k$ for some $l\in\{i,j\}$ (\emph{ordered upward-stability});
	\item[9b.] if $\langle i,j\cd C\rangle\in \mathcal{J}$ then $\langle i,j\cd C\setminus\{k\}\rangle\in \mathcal{J}$ for every $k\in V\setminus\{i,j\}$ such that $l\not< k$ for every $l\in \{i,j\}\cup C\setminus\{k\}$ (\emph{ordered downward-stability}).
\end{itemize}
As a consequence of Proposition \ref{prop:3}, we have the following:
\begin{corollary}\label{prop:2n}
For a directed acyclic graph $G$, the induced independence model by $d$-separation $\mathcal{J}(G)$ satisfies ordered upward- and downward-stability (8b) and (9b) w.r.t. the minimal order for $G$.
\end{corollary}
As a consequence of Theorem \ref{thm:0}, and for $G(\mathcal{J},\leq)$, the DAG obtained by directing the edges of $\sk(\mathcal{J})$ based on the order $\leq$, we have the following:
\begin{corollary}
Suppose, for an independence model $\mathcal{J}$ over $V$, that
\begin{enumerate}
  \item $\mathcal{J}$ is a compositional graphoid; and
  \item  there exists a $\mathcal{J}$-compatible order $\leq$ over $V$ w.r.t.\ which $\mathcal{J}$ satisfies ordered downward- and upward-stability, (8b) and (9b).
\end{enumerate}
Then $\mathcal{J}$ is minimally Markov to the directed acyclic graph $G(\mathcal{J},\leq)$.
\end{corollary}
The main conditions for an independence model to be faithful to a DAG (consequence of Theorem \ref{thm:2}) is as follows:
\begin{corollary}
Let $\mathcal{J}$ be an independence model defined over $V$. It then holds that $\mathcal{J}$ is faithful to a directed acyclic graph $G$ if and only if
\begin{enumerate}
  \item $\mathcal{J}$ is a singleton-transitive compositional graphoid; and
  \item there exists a $\mathcal{J}$-compatible order $\leq$ over $V$ w.r.t.\ which $\mathcal{J}$ satisfies ordered downward- and upward-stability, (8b) and (9b).
\end{enumerate}
In addition, if the conditions hold then $\mathcal{J}$ is faithful to $G(\mathcal{J},\leq)$.
\end{corollary}
For a probability distribution $P$, the first condition of the above corollary simplifies to that $P$ satisfies intersection, composition, and singleton-transitivity.
\section{Comparison to the Results in the Literature}\label{sec:res}
All results in the literature concerning faithfulness of certain probability distributions and certain types of graphs can be considered corollaries (or examples) of Theorem \ref{thm:2}. Here we verify this for some of the more well-known or interesting results. For brevity, we leave out some interesting results such as faithfulness of Gaussian and discrete distributions and LWF chain graphs; see \citet{pen11} and \citet{pen09,stub98} respectively.
\subsection{Faithfulness of $\mtp$ Distributions and Undirected Graphs} \emph{Multivariate totally positive of order $2$} ($\mtp$) distributions \citep{kar80} are distributions whose density $f$ is $\mtp$, that is $f(x)f(y)\leq f(x\wedge y)f(x\vee y)$, for all vectors $x,y$, where $x\wedge y=(\min(x_{1},y_{1}),\ldots,\min(x_{m},y_{m}))$ and  $x\vee y=(\max(x_{1},y_{1}),\ldots,\max(x_{m},y_{m}))$. In \citet{fal15}, it was shown that when an $\mtp$ distribution $P$ is a graphoid, it holds that $P$ and what we denote here by $G_u(P)$ are faithful. It is known that $\mtp$ distributions satisfy both singleton-transitivity and upward-stability, but not necessarily the intersection property; see \citet{fal15}. Hence, this result is a direct implication of Corollary \ref{coro:und}.
\subsection{Faithfulness of Gaussian Distributions and Undirected Graphs} \citet{lne07} provides regular Gaussian distributions that are faithful to any undirected graph; see Corollary 3 of the mentioned paper. In order to do so, given an undirected graph $G=(V,E)$, they provide the following matrix $A^{G,\epsilon}$, which acts as the covariance matrix of a Gaussian distribution that is faithful to $G$:
$$A^{G,\epsilon}=\left\{
    \begin{array}{ll}
1 &  \text{ if } i=j,\\
\epsilon & \text{ if } i\neq j  \text{ and $i$ and $j$ are adjacent in $G$},\\
0 & \text{ otherwise,}
    \end{array}
  \right.$$
where $i,j\in V$.

It is known that, for sufficiently small $\epsilon>0$, the corresponding Gaussian distribution is regular, and by Corollary \ref{coro:undgau}, we only need to show that it satisfies upward-stability. It is straightforward to show that for sufficiently small $\epsilon>0$, the inverse of $A^{G,\epsilon}$, which is the concentration matrix of the Gaussian distribution, is an \emph{$M$-matrix}, that is all its on-diagonal elements are positive and all its off-diagonal elements are non-positive. This is equivalent to the Gaussian distribution being $\mtp$; see \citet{kar83}. Hence, upward-stability is satisfied as discussed above.
\subsection{Unfaithfulness of Certain Gaussian Distributions and Undirected Graphs and DAGs} There are many examples in the literature for, and in fact it is easy to construct examples of, probability distributions that are not faithful to undirected graphs; for one of the first observations of this behaviour in a data set, see  Example 4 of \citet{cox93}. As a constructed example, let us discuss Example 1 of \citet{soh14}, which provided the Gaussian distribution for $X = (X_1 , X_2 , X_3 , X_4)$ with zero mean and positive definite covariance matrix
$$
\Sigma=\left[ {\begin{array}{cccc}
      3 & 2 & 1 & 2\\
      2 & 4 & 2 & 1\\
      1 & 2 & 7 & 1\\
      2 & 1 & 1 & 6
      \end{array} } \right].
$$
It can be seen that the concentration matrix $\Sigma^{-1}$ has no zero entries, but $X_1\ci X_3 \cd X_2$ since $\Sigma_{13} = \Sigma_{12}\Sigma^{-1}_{22}\Sigma_{23}$. Hence, the distribution and its complete skeleton are not faithful. This is clearly a consequence of the violation of upward-stability.

For DAGs too, it is easy to construct or observe examples where faithfulness is violated. One of the simplest of such examples, is what is related to the so-called ``transitivity of causation" in philosophical causality; see, for example, \citet{ram06}, where the goal is to show how the PC algorithm \citep{spioo} errs. Consider the graph $a\fra b\fra c$, where $a\ci c$ and $a\ci c\cd b$. This, in our language, is a clear violation of singleton-transitivity since the graph implies that $a,b$ and $b,c$ are always dependent.
\subsection{Faithfulness of Gaussian and Discrete Distributions and DAGs} It was shown in \citet{mee95} that, in certain
measure-theoretic senses, almost all the regular Gaussian and discrete distributions
that are Markov to a given DAG are faithful to it. It is known, for DAGs, that the global Markov property is equivalent to the \emph{factorization property} (see \citealt{lau96}), which holds if the joint density can factorize as follows:
$$f_V(x)=\prod_{v\in V}f_V(x_v\cd x_{\pa(v)}).$$
Since Gaussian and discrete distributions satisfy singleton-transitivity, in order to prove the result, one needs to show that ordered upward- and downward-stabilities, (8b) and (9b), are satisfied for almost all distributions that factorize w.r.t.\ the DAG. We refrain from showing this explicitly here due to technicality of the proof, but the general idea is similar to the proofs in Appendix B of \citet{mee95}: For (9b), we have that the probabilistic conditional independence holds for the marginal probability over $(i,j,C,k)$, and we need to exploit the factorization to show that the same conditional independence formula holds for the marginal probability over $(i,j,C)$ when $k\notin\an(\{i,j\}\cup C$. In this case the corresponding resulting polynomials are much simpler than those of \citet{mee95} since one needs to consider extra marginalization over only one variable $k$. For (8b), we will need to prove the other direction in a similar fashion.
\subsection{Faithfulness of Gaussian Distributions and Ancestral Graphs} The main goal of \citet{ric02} is to capture the conditional independence structure of a DAG model after marginalization and conditioning; see also \citet{sad13}. This was done by introducing the class of  AGs, and it was shown that every AG is probabilistic. In fact, there is a Gaussian distribution faithful to any AG, which is indeed the Gaussian distribution faithful to a DAG after marginalization and conditioning; see Theorem 7.5 of this paper.

However, by using our results, without going through the standard theory of ancestral graphs, it is possible to show that distributions that are faithful to a DAG are graphical after marginalization and conditioning: In order to do so, we use the notation in \citet{sad13}, for an independence model $\mathcal{J}$, to define the \emph{independence model after marginalization over $M$ and conditioning on $C$}, denoted by $\alpha(\mathcal{J};M,C$), for disjoint subsets $M$ and $C$ of $V$:
\begin{displaymath}
\alpha(\mathcal{J};M,C)=\{\langle A,B\cd D\rangle:\langle A,B\cd D\cup C\rangle\in \mathcal{J}\text{ and }(A\cup B\cup D)\cap (M\cup C)=\varnothing\}.
\end{displaymath}
First, we show that $\alpha(\mathcal{J};M,C$) satisfies condition 1 of Theorem \ref{thm:2}.
\begin{proposition}\label{prop:17}
 If $\mathcal{J}$ satisfies singleton-transitivity, intersection, and composition then so does $\alpha(\mathcal{J};M,C)$.
\end{proposition}
\begin{proof}
For composition, assume that $\langle A,B\cd C'\rangle\in\alpha(\mathcal{J};M,C)$ and $\langle A,D\cd C'\rangle\in\alpha(\mathcal{J};M,C)$. By definition, we have $\langle A,B\cd C'\cup C\rangle\in\mathcal{J}$ and $\langle A,D\cd C'\cup C\rangle\in\mathcal{J}$. Composition for $\mathcal{J}$ implies  $\langle A,B\cup D\cd C'\cup C\rangle\in\mathcal{J}$. Therefore, $\langle A,B\cup D\cd C'\rangle\in\alpha(\mathcal{J};M,C)$. Proof of the others is similar to this one and as straightforward.
\end{proof}
Denote by $\alpha(G;M,C)$, the generated AG from the DAG $G$ and $M$ and $C$ two disjoint subsets of its node set that are marginalized over and conditioned on respectively. Here we do not go through the generating process of AGs from DAGs; for more details, see \citet{ric02,sad13}.  Without explicitly using the theory of stability under marginalization and conditioning for AGs, we show the following.
\begin{proposition}
If $\mathcal{J}$ and a DAG $G$ are faithful then $\alpha(\mathcal{J};M,C)$ and the ancestral graph $\alpha(G;M,C)$ are faithful.
\end{proposition}
\begin{proof}
By Proposition \ref{prop:17}, we only need to show that w.r.t.\ the ordering associated with $\alpha(G;M,C)$, condition 2 of Theorem  \ref{thm:2} is also satisfied by $\alpha(\mathcal{J};M,C)$.

To prove ordered upward-stability, assume that  $\langle i,j\cd C'\cup\{k\}\rangle\notin\alpha(\mathcal{J};M,C)$, where $k$ is an anterior of a node $l$, $l\in\{i,j\}$ or connected by lines to $l$, $l\in C'$ in $H=\alpha(G;M,C)$. We have that $\langle i,j\cd C\cup C'\cup\{k\}\rangle\notin\mathcal{J}$. Since $\mathcal{J}$ and $G$ are faithful, there is a connecting walk $\pi$ between $i$ and $j$ given $C\cup C'\cup\{k\}$. If $k$ is not on $\pi$, $\pi$ is connecting given $C\cup C'$. Otherwise, it holds that $k$ must be a collider on $\pi$. By Lemma 2 in \citet{sad13}, $k$ is in ancestor of a node $l'\in\{i,j\}\cup C\cup C'$. By replacing the subwalk of $\pi$ from $l'$ to $k$ by the directed path from $k$ to $l'$ if $l'\in\{i,j\}$, or by adding the directed path from $k$ to $l'$ plus the reverse from $l'$ to $k$ to $\pi$ if $l'\in C\cup C'$, we conclude that  $\langle i,j\cd C\cup C'\rangle\notin\mathcal{J}$. Hence, $\langle i,j\cd C'\rangle\notin\alpha(\mathcal{J};M,C)$.


To prove ordered downward-stability, assume that $\langle i,j\cd C'\setminus\{k\}\rangle\notin\alpha(\mathcal{J};M,C)$, where $k\in C'$ is not an anterior of any node $l\in \{i,j\}\cup C'$ in $H$.  We have that $\langle i,j\cd C\cup C'\setminus\{k\}\rangle\notin\mathcal{J}$. Since $\mathcal{J}$ and $G$ are faithful, there is a connecting walk $\pi$ between $i$ and $j$ given $C\cup C'\setminus\{k\}$. Node $k$ cannot be on $\pi$ since $k$ must be a non-collider and consequently an ancestor of a node in $\{i,j\}\cup C\cup C'$ in $G$, which by Lemma 3 in \citet{sad13}, implies that $k\in\an(l)$ in $H$. Hence, $k$ is not on $\pi$, and consequently $\pi$ is connecting given $C\cup C'$.  We conclude that  $\langle i,j\cd C\cup C'\rangle\notin\mathcal{J}$. Hence, $\langle i,j\cd C'\rangle\notin\alpha(\mathcal{J};M,C)$.
\end{proof}
This shows that, for every AG, there is a distribution faithful to it since, for every AG, there is a DAG that, after marginalization and conditioning, results in the given AG.
\section{Summary and Future Work}\label{sec:dis}
We provided sufficient and necessary conditions for an independence model, and consequently a probability distribution, to be faithful to a chain mixed graph. All the definitions and results concerning independence models are true for probabilistic independence models. The conditions can be divided into two main categories: 1) Those that ensure that the distribution is Markov to a graph whose skeleton is the same as the skeleton of the distribution --- these properties are namely intersection, composition, and ordered upward- and downward-stabilities. If the distribution is already pairwise Markov to the graph then intersection and composition suffice. Ordered upward- and downward-stabilities direct the edges of the skeleton of the distribution so that the pairwise Markov property is satisfied. 2) Those that ensure that a Markov distribution is also faithful to the graph. These are singleton-transitivity and ordered upward- and downward-stabilities.

The type of preordering or preorderings w.r.t.\ which the distribution satisfies ordered upward- and downward-stabilities implies the type or types of graphs that are faithful to the distribution. In order to only deal with simpler classes of graphs, it is enough to search through a more refined set of preorderings determined by the conditions that graphs of a specific class must satisfy. It is indeed true that when the skeleton of the graph or distribution is known, the preordering is equivalent to different ways the edges of the skeleton can be directed, but in principle the preordering of the variables is unrelated to the skeleton of the distribution.

The conditions of faithfulness provided in this paper for probability distributions are sufficient and necessary. Thus, based on these conditions  many families of distributions can be shown not to be suitable for graphical modeling. This also shows the importance of Gaussian distributions in graphical models as the provided conditions (other than ordered upward- and downward-stabilities) are clearly satisfied by the regular Gaussian distributions. In addition, these conditions help those who devise new parametric models to ensure that their models satisfy the provided conditions for faithfulness.

It is clearly important to study the characteristics of distributions that satisfy these conditions. The intersection property is well-studied and, in addition to positivity of the density that has been known to be a sufficient condition, necessary and sufficient conditions for a probabilistic independence to satisfy this property are known; see \citet{pet15}. For composition and singleton-transitivity, more studies are needed; in particular, it would be worthwhile to find simple sufficient conditions for probability distributions in order to satisfy these properties.

Another approach is to devise algorithms that test whether these conditions are satisfied given data or an independence model. Indeed one can argue for an alternative approach to
determine whether a distribution is faithful to some graph: one first applies the structure learning
algorithm (taking as input the distribution and outputting a graph) and then one tests whether the
separation relationships in the found graph correspond to the conditional indepedencies in the
distribution. However, testing our conditions are much more efficient. As an example, in the case of undirected graphs and Gaussian distributions, by our results, one only needs to test upward-stability for non-adjacent single elements $i$ and $j$ (starting with a minimal node cut-set as the conditioning set) rather than testing all global conditional independencies implied by the graph.

As another algorithmic consequence that the results in this paper may prompt, they lead to a \emph{constraint-based} structural learning algorithm (as opposed to score-based, for example, \citet{chi03}) most similar to the Fast Causal Inference (FCI) algorithm \citep{spioo} --- it finds the skeleton of the graph and then explores directions of the edges of the graph on the skeleton. Directing the edges of the skeleton is performed based on an exploitation of the preordering of the variables (and consequently the nodes of the graph) so that w.r.t.\ such a preordering ordered upward- and downward-stabilities are satisfied. This algorithm would have two advantages over the known constraint-based algorithms in the literature: Firstly, it is devised for a larger class of AnGs, which can be thought of as learning the LWF chain graphs with hidden and selection variables; see \citet{sad16}. Secondly, the exploitation of ordered upward- and downward-stability for structural learning can be combined with actually testing whether these conditions are satisfied, and the theory in the paper ensures the faithfulness of the generated graphs. We plan on providing the details of this algorithm elsewhere.

When dealing with data, another issue is that obtaining a faithful graph is very
sensitive to hypothesis testing errors for inferring the independence statements from data; see \citet{uhl13}. The concept of \emph{strong faithfulness} \citep{zha03} has been suggested for the Gaussian case in order to deal with this issue. Although strong faithfulness, by nature, does not seem to be characterizable in this way, it is an interesting problem to study which of the equivalent conditions to faithfulness are more prone to such sensitivity.



\acks{The original idea of this paper was inspired by discussions at the
American Institute of Mathematics workshop ``Positivity, Graphical Models,
and the Modeling of Complex Multivariate Dependencies'' in October 2014.
The author thanks the AIM and NSF for their support of the workshop and is indebted to
other participants of the discussion group, especially to Shaun Fallat, Steffen Lauritzen, Nanny Wermuth, and Piotr Zwiernik for follow-up discussions. The author is also grateful to Rasmus Petersen for valuable discussions. This work was partly supported by grant $\#$FA9550-12-1-0392 from the U.S. Air Force Office of Scientific Research (AFOSR) and the Defense Advanced Research Projects Agency (DARPA).}



\vskip 0.2in
\bibliography{bib}

\end{document}